\pdfoutput=1
\documentclass[11pt,a4paper]{amsart}
\usepackage[margin=1in]{geometry}
\usepackage[T1]{fontenc}
\usepackage{textcase}
\usepackage[dvipsnames]{xcolor}
\usepackage{microtype}
\usepackage{fnpct}

\usepackage{amsmath}
\usepackage{amssymb}
\usepackage{eucal}
\usepackage{mathrsfs}
\usepackage{tikz-cd}

\usepackage{enumitem}
\usepackage{booktabs}
\usepackage[pdfusetitle,colorlinks]{hyperref}
\hypersetup{bookmarksdepth=2,pdfencoding=unicode,allcolors=MidnightBlue}

\usepackage{zref-clever}
\zcsetup{abbrev=false,cap=true,nameinlink=false,sort=false,lang=english}
\newcommand{\cref}[1]{\zcref{#1}}
\newcommand{\Cref}[1]{\zcref[S]{#1}}
\zcsetup{pairsep={ and~},lastsep={, and~}}
\zcRefTypeSetup{equation}{Name-sg=,Name-pl=,refbounds={(,,,)}}
\AddToHook{env/equation/begin}{\zcsetup{countertype={equation=equation}}}
\zcRefTypeSetup{item}{Name-sg=,Name-pl=,refbounds={(,,,)}}
\newlist{conenum}{enumerate}{1}
\setlist[conenum,1]{label=(\roman*),ref=\roman*}
\zcRefTypeSetup{conenumi}{Name-sg=,Name-pl=,refbounds={(,,,)}}

\NewDocumentCommand{\newzctheorem}{momo}{\IfValueTF{#4}
  {\newtheorem{#1}{#3}[#4]}
  {\IfValueTF{#2}
    {\AddToHook{env/#1/begin}{\zcsetup{countertype={#2=#1}}}\newtheorem{#1}[#2]{#3}}
    {\newtheorem{#1}{#3}}}}
\numberwithin{equation}{section}
\theoremstyle{plain}
\newtheorem*{Theorem}{Main Theorem}
\newzctheorem{theorem}[equation]{Theorem}
\newzctheorem{proposition}[equation]{Proposition}
\newzctheorem{lemma}[equation]{Lemma}
\newzctheorem{corollary}[equation]{Corollary}
\theoremstyle{definition}
\newzctheorem{definition}[equation]{Definition}
\newzctheorem{example}[equation]{Example}
\theoremstyle{remark}
\newzctheorem{remark}[equation]{Remark}

\let\oldSS\SS\let\SS\relax

\newcommand{\NN}{\mathbf{N}}
\newcommand{\QQ}{\mathbf{Q}}

\newcommand{\E}{\mathrm{E}}

\newcommand{\Sh}{\operatorname{Sh}}
\newcommand{\Cor}{\operatorname{Cor}}

\newcommand{\sgn}{\operatorname{sgn}}
\newcommand{\GL}{\operatorname{GL}}
\newcommand{\st}{\operatorname{st}}

\newcommand{\SS}{\operatorname{S}}
\newcommand{\ST}{\mathscr{T}}

\newcommand{\Lie}[1]{\mathcal{#1}}

\providecommand*\shuffle{\mathbin{\mathpalette\shuffleaux\relax}}
\newcommand*\shuffleaux[2]{\sbox0{$#1\vcenter{}$}\kern.15\ht0
\rlap{\rule{2.5\ht0}{.25\ht0}}\hbox to2.5\ht0{\def\x{\rule[.2\ht0]{.17\ht0}{1.65\ht0}}\x\hfil\x\hfil\x}\kern.15\ht0}

\title[The decomposition relations]{The decomposition relations alone\texorpdfstring{\\}{ }present the Goncharov Lie coalgebra}
\author{Ko Aoki}
\address{Department of Mathematical Sciences,
  University of Copenhagen, Denmark
}
\email{aoki@math.ku.dk}
\date{\today}

\begin{document}

\begin{abstract}
  Recently,
  Kupers–Rudenko–Sierra
  introduced the Goncharov Lie coalgebra.
  They gave a presentation using correlator generators
  and four types of relations:
  homogeneity,
  cyclic symmetry,
  shuffle relations,
  and decomposition relations.
  We show that the first three families of relations
  follow from the decomposition relations.
\end{abstract}

\maketitle
\setcounter{tocdepth}{1}

\section{Introduction}\label{s:intro}

For a field~\(F\),
Kupers–Rudenko–Sierra~\cite{KRS1}
introduced the \emph{Goncharov Lie coalgebra}
in terms of \(\E_{\infty}\)-homology
as
\begin{equation*}
  \Lie{G}(F)
  =\bigoplus_{n\geq1}\Lie{G}_{n}(F)
  =\bigoplus_{n\geq1}H_{n,2n-1}^{\E_{\infty}}(B{\GL_{*}(F)};\QQ).
\end{equation*}
They proved in~\cite{KRS2} that
when \(F\) is a number field,
the finite-dimensional comodule category
over this graded Lie coalgebra
is equivalent to the category of mixed Tate motives over~\(F\).
We study their \emph{correlator} presentation of \(\Lie{G}(F)\)\footnote{They gave a formula for the cobracket in~\cite[Theorem~C.b]{KRS1}.
}
from~\cite[Theorem~C.a]{KRS1}:

\begin{theorem}[Kupers–Rudenko–Sierra]\label{xvi80j}
  For \(n\geq1\),
  the \(\QQ\)-vector space~\(\Lie{G}_{n}(F)\)
  has a presentation
  with generators
  \(\Cor(x_{0},\dotsc,x_{n})\)\footnote{We omit the superscript~\(\Lie{G}\)
    here, since other types of correlators are not used in this paper.
  },
  where \(x_{0}\), \dots, \(x_{n}\in F\) are not all equal,
  and relations as follows:
  \begin{description}
    \item[Homogeneity]
      \(\Cor(x_{0},\dotsc,x_{n})=\Cor(x_{0}+b,\dotsc,x_{n}+b)\)
      for \(b\in F\).
    \item[Cyclic symmetry]
      \(\Cor(x_{0},x_{1},\dotsc,x_{n})=\Cor(x_{1},x_{2},\dotsc,x_{n},x_{0})\).
    \item[Shuffle relations]
      For \(n=p+q\) with~\(p\), \(q>0\),
      \begin{equation*}
        \sum_{\sigma\in\Sh(p,q)}
        \Cor(x_{0},x_{\sigma(1)},\dotsc,x_{\sigma(n)})=0.
      \end{equation*}
    \item[Decomposition relation]
      We recall
      in \cref{s:com} the definition of \(\ST(n)\) and \(\sgn\).
      \begin{equation*}
        \begin{split}
          &\Cor(x_{0},\dotsc,x_{n})-\Cor(y_{0},\dotsc,y_{n}) \\
          &=
          \sum_{T=(i_{1}j_{1},\dotsc,i_{n}j_{n})\in\ST(n)}
          \sgn(T)
          \Cor\left(
            0,
            \frac{x_{i_{1}}-x_{j_{1}}}{y_{i_{1}}-y_{j_{1}}},
            \dotsc,
            \frac{x_{i_{n}}-x_{j_{n}}}{y_{i_{n}}-y_{j_{n}}}
          \right),
        \end{split}
      \end{equation*}
      where a summand is omitted if
      \(y_{i_{k}}=y_{j_{k}}\) for some \(k\).
  \end{description}
\end{theorem}

In~\cite[Remark~7.2]{KRS1},
Kupers–Rudenko–Sierra noted that homogeneity is redundant
and raised the analogous question for cyclic symmetry.
We answer this affirmatively and show, moreover,
that the shuffle relations are redundant:

\begin{Theorem}\label{xwzp4u}
  For \(n\geq1\),
  the \(\QQ\)-vector space~\(\Lie{G}_{n}(F)\)
  is presented by
  \(\Cor(x_{0},\dotsc,x_{n})\),
  where \(x_{0}\), \dots, \(x_{n}\in F\) are not all equal,
  and the decomposition relation.
\end{Theorem}

We prove this in \cref{s:proof}.
The proof is a symbolic manipulation of the decomposition relation.
In \cref{s:com},
we recall the universal symbols
and isolate the star terms used in the proof.

\subsection*{Acknowledgments}\label{ss:ack}

During the course of this work
I was supported by the Danish National Research Foundation
through the Copenhagen Center for Geometry and Topology (DNRF151).

\section{The universal symbols}\label{s:com}

For distinct \(s\), \(s'\in\NN\),
we write \(ss'\) for the two-element subset \(\{s,s'\}\);
we regard it as an edge
of the complete graph on~\(\NN\).
We recall the following from~\cite[Section~2.6]{KRS1}:

\begin{definition}\label{xui4oi}
  For \(n\geq1\)
  and distinct formal indices \(s_{0}\), \dots, \(s_{n}\in\NN\),
  the \emph{universal symbol} \(\SS_{n}(s_{0},\dotsc,s_{n})\)
  in the tensor algebra over \(\QQ\) on \(\binom{\NN}{2}\)
  is defined recursively by
  \(
    \SS_{1}(s_{0},s_{1})=s_{0}s_{1}
  \) and, for \(n\geq2\), by
  \begin{equation*}
    \SS_{n}(s_{0},\dotsc,s_{n})
    =
    \sum_{i=1}^{n}
      \SS_{n-1}(s_{0},\dotsc,\widehat{s_{i}},\dotsc,s_{n})
      \otimes s_{i}s_{i+1}
    -
    \sum_{i=1}^{n-1}
      \SS_{n-1}(s_{0},\dotsc,\widehat{s_{i+1}},\dotsc,s_{n})
      \otimes s_{i}s_{i+1},
  \end{equation*}
  where \(s_{n+1}=s_{0}\).
\end{definition}

Each tensor monomial in the expansion of
\(\SS_{n}(0,\dotsc,n)\)
records an ordered \(n\)-tuple of edges.
Induction on the recurrence shows that
the underlying edge set is a spanning tree
of the complete graph on \(\{0,\dotsc,n\}\),
and that no ordered monomial occurs more than once.
Consequently, we have the following:

\begin{definition}\label{xz6z6l}
  For \(n\geq1\),
  let \(\ST(n)\) be the unique subset
  of the set of ordered \(n\)-tuples of edges
  of the complete graph on \(\{0,\dotsc,n\}\)
  and
  \(\sgn\colon\ST(n)\longrightarrow\{\pm1\}\)
  the unique map such that
  \begin{equation*}
    \SS_{n}(0,\dotsc,n)
    =\sum_{T=(i_{1}j_{1},\dotsc,i_{n}j_{n})\in\ST(n)}
    \sgn(T)
    (i_{1}j_{1}\otimes\dotsb\otimes i_{n}j_{n}).
  \end{equation*}
\end{definition}

Note that the ordering is essential;
e.g.,
\(\ST(3)\) contains
\((01,03,12)\) and
\((01,12,03)\)
as distinct elements;
cf.~\cite[Section~7.3.2]{KRS1}.

For the proof of the main theorem,
we only need the terms indexed by stars.
An ordered spanning tree is a \emph{star centered at~\(j\)}
if every edge is incident with~\(j\).
When \(n=1\), the unique edge may be regarded as a star
centered at either endpoint.
The possible orderings of a star
appearing in \(\ST(n)\)
are encoded by shuffle products.
For two tensors~\(L\) and~\(R\)
in the tensor algebra,
we write \(L\shuffle R\)
for their \emph{shuffle product},
i.e., the sum of all order-preserving interleavings
of their tensor factors;
e.g.,
\begin{equation*}
  (e\otimes e')\shuffle f
  =e\otimes e'\otimes f+e\otimes f\otimes e'+f\otimes e\otimes e'.
\end{equation*}
For \(1\leq j\leq n\), consider
\begin{align*}
  L_{j}&=s_{1}s_{j}\otimes\dotsb\otimes s_{j-1}s_{j},&
  R_{j}&=s_{n}s_{j}\otimes\dotsb\otimes s_{j+1}s_{j},
\end{align*}
where an empty tensor is interpreted as the unit.
Let
\(\st_{j}\SS_{n}(s_{0},\dotsc,s_{n})\)
denote the sum of the monomials
whose underlying tree is the star centered at~\(s_{j}\).

\begin{lemma}\label{star}
  For \((n,j)\) with \(1\leq j\leq n\),
  we have
  \begin{equation*}
    \st_{j}\SS_{n}(s_{0},\dotsc,s_{n})
    =(-1)^{n+j}s_{0}s_{j}
      \otimes\bigl(L_{j}\shuffle R_{j}\bigr).
  \end{equation*}
\end{lemma}

\begin{proof}
  We argue by induction on~\(n\).
  When \(n=1\),
  this follows from the definition.
  We suppose \(n\geq2\) and the assertion holds for \(n-1\).

  When \(j=1\),
  the only relevant summand in the recurrence for
  \(\SS_{n}(s_{0},\dotsc,s_{n})\) is
  \begin{equation*}
    -\SS_{n-1}(s_{0},s_{1},s_{3},\dotsc,s_{n})
    \otimes s_{1}s_{2}.
  \end{equation*}
  This case follows
  from the inductive hypothesis applied to
  \((n,j)=(n-1,1)\).

  When \(1<j<n\),
  the relevant summands in the recurrence for
  \(\SS_{n}(s_{0},\dotsc,s_{n})\)
  are
  \begin{equation*}
    \SS_{n-1}(s_{0},\dotsc,\widehat{s_{j-1}},\dotsc,s_{n})
    \otimes s_{j-1}s_{j}
    -\SS_{n-1}(s_{0},\dotsc,\widehat{s_{j+1}},\dotsc,s_{n})
    \otimes s_{j}s_{j+1}.
  \end{equation*}
  By the inductive hypothesis applied to \((n,j)=(n-1,j-1)\),
  the former produces the shuffles in \(L_{j}\shuffle R_{j}\)
  ending in \(s_{j-1}s_{j}\).
  By the inductive hypothesis applied to \((n,j)=(n-1,j)\),
  the latter produces precisely those
  ending in \(s_{j+1}s_{j}\).
  These two classes partition all the shuffles.
  The signs of the two classes are, respectively,
  \((-1)^{(n-1)+(j-1)}\) and
  \(-(-1)^{(n-1)+j}\),
  both equal to \((-1)^{n+j}\).

  When \(j=n\),
  the only relevant summand in the recurrence for
  \(\SS_{n}(s_{0},\dotsc,s_{n})\)
  is
  \begin{equation*}
    \SS_{n-1}(s_{0},\dotsc,s_{n-2},s_{n})
    \otimes s_{n-1}s_{n}.
  \end{equation*}
  This case follows from the inductive hypothesis applied to
  \((n,j)=(n-1,n-1)\).
\end{proof}

\section{Proof}\label{s:proof}

We fix a field~\(F\) and \(n\geq1\);
we do not induct on~\(n\).
We work in
the \(\QQ\)-vector space generated by symbols
\(C(x_{0},\dotsc,x_{n})\),
where \(x_{0}\), \dots, \(x_{n}\in F\) are not all equal,
subject to the decomposition relation.
We prove
the homogeneity,
cyclic symmetry,
and shuffle relations
in \cref{hom,cyc,shuffle},
respectively.

\begin{proposition}\label{hom}
  For \(b\in F\),
  we have \(C(x_{0}+b,\dotsc,x_{n}+b)=C(x_{0},\dotsc,x_{n})\).
\end{proposition}

\begin{proof}
  We write
  \(x=(x_{0},\dotsc,x_{n})\) and
  \(x+b=(x_{0}+b,\dotsc,x_{n}+b)\).
  Comparing the decomposition relations for the pairs
  \((x,x+b)\) and \((x,x)\),
  we see \(C(x)-C(x+b)=C(x)-C(x)=0\).
\end{proof}

\begin{proposition}\label{cyc}
  We have
  \(C(x_{0},x_{1},\dotsc,x_{n})=C(x_{1},\dotsc,x_{n},x_{0})\).
\end{proposition}

\begin{proof}
  We write \(x=(x_{0},\dotsc,x_{n})\)
  and \(\rho(x)=(x_{1},\dotsc,x_{n},x_{0})\).
  We consider
  \(u=(1^{n},0)\in F^{n+1}\)
  and apply the decomposition relation to the pair \((x,u)\).
  The denominator corresponding
  to an edge~\(ij\)
  vanishes when \(i\), \(j<n\).
  Thus every relevant ordered tree
  has underlying tree the star centered at~\(n\).
  We can now use \cref{star} for \(j=n\),
  which says \(\st_{n}\SS_{n}(0,\dotsc,n)=0n\otimes\dotsb\otimes(n-1)n\).
  Therefore, we have
  \begin{equation}
    \label{e:4wdzl}
    C(x)-C(u)
    =C(0,x_{0}-x_{n},\dotsc,x_{n-1}-x_{n})
    =C(x_{n},x_{0},\dotsc,x_{n-1})
    =C(\rho^{-1}x),
  \end{equation}
  where we used \cref{hom}.

  Summing \cref{e:4wdzl}
  after replacing \(x\) successively by
  \(x\), \(\rho x\), \dots, \(\rho^{n}x\),
  gives \((n+1)C(u)=0\),
  and hence \(C(u)=0\).
  Therefore \(C(x)=C(\rho^{-1}x)\).
  Iterating this identity gives
  \(C(x)=C(\rho^{-n}x)=C(\rho x)\),
  which is the desired cyclic symmetry.
\end{proof}

\begin{proposition}\label{shuffle}
  For \(n=p+q\) with~\(p\), \(q>0\),
  we have
  \begin{equation*}
    \sum_{\sigma\in\Sh(p,q)}
      C(x_{0},x_{\sigma(1)},\dotsc,x_{\sigma(n)})=0.
  \end{equation*}
\end{proposition}

\begin{proof}
  For finite tuples~\(A\) and~\(B\),
  we write \(\Sh(A,B)\)
  for the family of their order-preserving interleavings
  (counted with multiplicity).

  We take a nonconstant tuple \((y_{0},\dotsc,y_{n})\)
  and consider \(-y=(-y_{0},\dotsc,-y_{n})\).
  For \(1\leq j\leq n\),
  we consider \(e(j)=(0^{j},1,0^{n-j})\).
  We write~\(D(j)=C(-y)-C(e(j))\).
  By \cref{cyc}, \(C(e(j))\) does not depend on~\(j\),
  and hence neither does~\(D(j)\).
  We then compute~\(D(j)\)
  using the decomposition relation for the pair \((-y,e(j))\).
  As in the proof of \cref{cyc},
  the only relevant part is
  \(\st_{j}\SS_{n}(0,\dotsc,n)\).
  By \cref{star},
  \((-1)^{n+j}D(j)\) is the sum indexed by
  \((k_{1},\dotsc,k_{n-1})\in\Sh((1,\dotsc,j-1),(n,\dotsc,j+1))\).
  For an edge \(kj\), the corresponding ratio is
  \begin{equation*}
    \frac{(-y_{k})-(-y_{j})}{e(j)_{k}-e(j)_{j}}
    =y_{k}-y_{j}.
  \end{equation*}
  The corresponding summand is
  \begin{equation*}
    C(0,y_{0}-y_{j},y_{k_{1}}-y_{j},\dotsc,y_{k_{n-1}}-y_{j})
    =C(y_{j},y_{0},y_{k_{1}},\dotsc,y_{k_{n-1}})
    =C(y_{0},y_{k_{1}},\dotsc,y_{k_{n-1}},y_{j}),
  \end{equation*}
  where we used \cref{hom,cyc}.
  Therefore, we have
  \begin{equation*}
    D(j)=(-1)^{n+j}
    \sum_{y'\in
      \Sh((y_{1},\dotsc,y_{j-1}),(y_{n},\dotsc,y_{j+1}))}
      C(y_{0},y',y_{j}).
  \end{equation*}

  We now return to \((x_{0},\dotsc,x_{n})\). We consider
  \begin{align*}
    A&=(x_{1},\dotsc,x_{p}),&
    A^{-}&=(x_{1},\dotsc,x_{p-1}),&
    B&=(x_{p+1},\dotsc,x_{n}),&
    B^{-}&=(x_{p+1},\dotsc,x_{n-1}).
  \end{align*}
  By applying the previous discussion to
  \begin{equation*}
    (y_{0},y_{1},\dotsc,y_{n})
    =(x_{0},x_{1},\dotsc,x_{p},x_{n},x_{n-1},\dotsc,x_{p+1}),
  \end{equation*}
  we have
  \begin{align*}
    D(p)&=(-1)^{n+p}
    \sum_{x'\in\Sh(A^{-},B)}C(x_{0},x',x_{p}),&
    D(p+1)&=(-1)^{n+p+1}
    \sum_{x'\in\Sh(A,B^{-})}C(x_{0},x',x_{n}).
  \end{align*}
  Since \(D(j)\) is independent of~\(j\),
  we have
  \begin{equation*}
    \sum_{x'\in\Sh(A^{-},B)}C(x_{0},x',x_{p})
    +\sum_{x'\in\Sh(A,B^{-})}C(x_{0},x',x_{n})=0.
  \end{equation*}
  Every \((p,q)\)-shuffle places
  either
  the final entry~\(x_{p}\) of~\(A\)
  or
  the final entry~\(x_{n}\) of~\(B\)
  last.
  Removing that last entry gives, respectively,
  a shuffle of~\(A^{-}\) with~\(B\)
  or of~\(A\) with~\(B^{-}\).
  Thus the two sums
  partition the \((p,q)\)-shuffles,
  proving the result.
\end{proof}

\bibliographystyle{plain}
\let\SS\oldSS  \newcommand{\yyyy}[1]{}

\end{document}